\documentclass[a4paper,12pt]{amsart}

\usepackage{amsmath,amssymb,amsthm,tikz,hyperref}
\usepackage[T1]{fontenc}
\usepackage[latin1]{inputenc}
\usetikzlibrary{matrix, calc, arrows}

\newtheorem{theorem}{Theorem}
\newtheorem{lemma}{Lemma}

\newcommand{\ZZ}{\mathbb{Z}}                        
\newcommand{\RR}{\mathbb{R}}                       
\newcommand{\CC}{\mathbb{C}}                       
\newcommand{\LR}[3]{c_{#1, #2}^{#3}}             
\newcommand{\GS}{S_2 \times (S_3 \wr S_2)} 
\newcommand{\Cone}{\mathcal{C}}                    
\newcommand{\rays}{\mathcal{R}}                    
\renewcommand{\S}{S}                    
\renewcommand{\|}{\;|\;}

\title[Symmetries of the Littlewood--Richardson coefficients]{The 144 symmetries of the Littlewood-Richardson coefficients of $SL_3$}
\date{\today}
\author{Emmanuel Briand}
\address{Emmanuel Briand, Departamento Matemática Aplicada I, Universidad de Sevilla}
\email{ebriand@us.es}
 \author{Mercedes Rosas}
\address{Mercedes Rosas, Departamento de \'Algebra, Universidad de Sevilla}
\email{mrosas@us.es}
\thanks{Both authors are partially supported by MTM2016-75024-P and FEDER, and Junta de Andalucia under grants P12-FQM-2696 and FQM-333.}

\begin{document}

\begin{abstract}
We compute with \emph{SageMath} the group of all linear symmetries for the Littlewood-Richardson associated to the representations of $SL_3$. We find that there are 144 symmetries, more than the 12 symmetries known for the Littlewood-Richardson coefficients in general.
\end{abstract}

\maketitle


\section{Introduction}

The \emph{Littlewood--Richardson coefficients} $\LR{\lambda}{\mu}{\nu}$ are among  the most studied  constants in geometry and representation theory. 
In geometry, they are the structure constants for the multiplication in the cohomology ring of the Grassmannian, in the basis of the Schubert cycles \cite{Fulton}. 
They can also be interpreted as cardinalities of the intersection of some triples of Schubert varieties \cite{BilleyVakil}. 
In the representation theory of the general linear group, they are the multiplicities in the tensor product of irreducible representations, and also the dimensions of the subspace of invariants in the triple tensor products of irreducible representations \cite{BerensteinZelevinsky}. 
In the representation theory of the symmetric groups, they are the multiplicities in the restrictions of the irreducible representations of a symmetric group $\S_{m+n}$ to its Young subgroup $\S_m \times \S_n$. 
In the theory of symmetric functions, they are the structure constants for the ordinary multiplication in the basis of Schur functions.  
They receive also a number a combinatorial interpretations: they  are known to count Littlewood-Richardson tableaux, hives, BZ triangles, \ldots

Together, these different interpretations of the Littlewood-Richardson coefficients make clear some symmetries they afford. 
For instance, the descriptions as structural constants for commutative  products  (of Schur functions, in the ring of symmetric functions; of Schubert cycles, in the cohomology ring of the Grassmannian) make clear the invariance 
$\LR{\lambda}{\mu}{\nu} = \LR{\mu}{\lambda}{\nu}$, under exchanging the two partitions $\lambda$ and $\mu$. The descriptions as cardinalities of intersection of triples of Schubert varieties \cite{ThomasYong}, or as dimensions of subspaces of invariants in triple tensor products \cite{BerensteinZelevinsky} reveals the existence of a $\S_3$--symmetry of the Littlewood--Richardson coefficients (that comprises the previous symmetry):
\begin{equation}\label{S3}
\LR{\lambda}{\mu}{\nu}
= \LR{\mu}{\lambda}{\nu}
= \LR{\mu}{\nu^{\square}}{\lambda^{\square}}
= \LR{\nu^{\square}}{\mu}{\lambda^{\square}}
= \LR{\nu^{\square}}{\lambda}{\mu^{\square}}
= \LR{\lambda}{\nu^{\square}}{\mu^{\square}}
\end{equation}
where $\lambda^{\square}$, $\mu^{\square}$ and $\nu^{\square}$ are obtained from the Young diagrams of ${\lambda}$, ${\mu}$ and ${\nu}$ by ``taking complements'' in suitable rectangles (see Figure \ref{figure:LR}).
\begin{figure}[ht]
\begin{tikzpicture}[scale=.3]
\tikzstyle{every node}=[inner sep=3pt]; 
\draw[thick] (1,.5)--(1,4.5)--(5,4.5)--(5,3.5)--(6.5,3.5)--(6.5,2.5)--(8,2.5)--(8,.5)--cycle;
\draw (3.5,2.5) node { ${\lambda} $};
\draw[dashed] (1,.5)--(1,6)--(8,6)--(8,2.5);
\draw (6.7,4.5) node {$\lambda^{\square}$};
\end{tikzpicture}
\caption{New Young diagram obtained as complement of a Young diagram in a rectangle. When considering the Littlewood--Richardson coefficients associated to $SL_N$. with $\lambda$ and $\mu$ of length at most $N-1$ and $\nu$ of length at most $N$, one takes complements in rectangles of height $N$ and: for \eqref{S3}, of length $\nu_1$; for \eqref{3rectangles}, of length respectively $\lambda_1$, $\mu_1$, $\lambda_1+\mu_1$ for $\lambda^{\square}$, $\mu^{\square}$ and $\nu^{\square}$.} 
\label{figure:LR}
\end{figure}
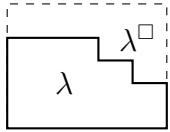

Duality provides an additional symmetry (see \cite{BriandOrellanaRosas:symmetries}):
\begin{equation}\label{3rectangles}
\LR{\lambda}{\mu}{\nu} = \LR{\lambda^{\square}}{\mu^{\square}}{\nu^{\square}}.
\end{equation}

Together, these ``known symmetries'' generate a group of $12$ symmetries, isomorphic to $\S_2 \times \S_3$. Note that it is not straightforward to understand these symmetries all together from the combinatorial descriptions of the Littlewood-Richardson coefficients (see \cite{PakVallejo:cones,  ThomasYong, TeradaKingAzenhas}).

\smallskip
Is it possible that the Littlewood--Richardson coefficients afford additional, unknown symmetries?

\smallskip

To settle this question, we adopt an experimental approach. We consider the family of Littlewood-Richardson coefficients  indexed by partitions $\lambda$, $\mu$ and $\nu$ with restricted length. The combinatorial object they count can be interpreted as lattice points in a (rational convex polyhedral) cone. As a consequence (see \cite{Rassart}), $\LR{\lambda}{\mu}{\nu}$ is a piecewise quasipolynomial function (actually polynomial, by an additional argument given again in \cite{Rassart}) of the parts 
of the partitions $\lambda$, $\mu$ and $\nu$. The domains of polynomiality are the maximal cones (``chamber'') of a fan (a complex of rational polyhedral convex cones, called the ``Chamber Complex''). Explicit knowledge of such a description provides data making possible an exhaustive search of all linear symmetries. 

The Chamber Complex and the polynomial formulas for the Littlewood--Richardson coefficients $\LR{\lambda}{\mu}{\nu}$ related to the representations of $GL_3(\CC)$ (i.e. where all three partitions have length at most $3$) were explicitly given in \cite{Rassart}. 
We exploit and analyze these data with \emph{SageMath} \cite{sagemath}, and obtain: 
\begin{theorem}\label{thm:main}
The group of linear symmetries of the Littlewood--Richardson coefficients $\LR{\lambda}{\mu}{\nu}$ associated to $SL_3$, has order 144. It is isomorphic to $\GS$, and acts transitively on the chambers of the Chamber Complex. It is generated by the well--known group of 12 symmetries, and the additional symmetry:
\[
\LR{(\lambda_1, \lambda_2)}{(\mu_1, \mu_2)}{(\nu_1, \nu_2, \nu_3)} 
=
\LR{(\lambda_1+\mu_1-\nu_2, \lambda_2+\mu_1-\nu_2)}{(\nu_2, \mu_2)}{(\nu_1, \mu_1, \nu_3)} 
\]
\end{theorem}

The calculations are presented in detail in a \emph{SageMath} Notebook available online \cite{calculations} and as an ancillary file.

\section{The Littlewood--Richardson coefficients associated to $SL_3$ and their known symmetries}

\subsection{The Littlewood--Richardson coefficients associated to $SL_3$}

In this work we restrict our study to the family of Littlewood-Richardson coefficients related to $SL_3$. These are the Littlewood--Richardson coefficients $\LR{\lambda}{\mu}{\nu}$ such that $\lambda$ and $\mu$ have length at most $2$ and $\nu$ has length at most $3$. 
Since $\LR{\lambda}{\mu}{\nu}=0$ when $|\lambda|+ |\mu| \neq |\nu|$, we have for any non--zero coefficient $\LR{(\lambda_1, \lambda_2)}{(\mu_1,\mu_2)}{(\nu_1,\nu_2,\nu_3)}$ that
\begin{equation}\label{nu3}
\nu_3 = \lambda_1 + \lambda_2 + \mu_1+\mu_2 - \nu_1 - \nu_2.
\end{equation}
We thus consider the function $C$ defined on $\ZZ^6$ by 
\[
C(\lambda_1, \lambda_2, \mu_1, \mu_2, \nu_1, \nu_2) = 
\LR{(\lambda_1, \lambda_2)}{(\mu_1,\mu_2)}{(\nu_1,\nu_2,\nu_3)}  
\]
when $\lambda_1 \ge \lambda_2 \ge 0$, $\mu_1 \ge \mu_2 \ge 0$, $\nu_1 \ge \nu_2 \ge \nu_3 \ge 0$ (with $\nu_3$ defined by \eqref{nu3}), and $C(\lambda_1, \lambda_2, \mu_1, \mu_2, \nu_1, \nu_2) = 0$ otherwise.

Note that considering instead  the Littlewood-Richardson coefficients associated to $GL_3$ (the $\LR{(\lambda_1, \lambda_2, \lambda_3)}{(\mu_1,\mu_2,\mu_3)}{(\nu_1,\nu_2,\nu_3)}$) gives equivalent results (but hides symmetries). This is because of the invariance properties:
\[
\LR{(\lambda_1, \lambda_2, \lambda_3)}{(\mu_1,\mu_2,\mu_3)}{(\nu_1,\nu_2,\nu_3)} = \LR{(1+\lambda_1, 1+\lambda_2, 1+\lambda_3)}{(\mu_1,\mu_2,\mu_3)}{(1+\nu_1,1+\nu_2,1+\nu_3)} 
\]
and
\[
\LR{(\lambda_1, \lambda_2, \lambda_3)}{(\mu_1,\mu_2,\mu_3)}{(\nu_1,\nu_2,\nu_3)} = \LR{(\lambda_1, \lambda_2, \lambda_3)}{(1+\mu_1,1+\mu_2,1+\mu_3)}{(1+\nu_1,1+\nu_2,1+\nu_3)} 
\]
which imply 
\[
\LR{(\lambda_1, \lambda_2, \lambda_3)}{(\mu_1,\mu_2,\mu_3)}{(\nu_1,\nu_2,\nu_3)} = \LR{(\lambda_1, \lambda_2)}{(\mu_1,\mu_2)}{(\nu_1-\mu_3-\lambda_3,\nu_2-\mu_3-\lambda_3,\nu_3-\mu_3-\lambda_3)}. 
\]

\subsection{Known symmetries}

We now look for all linear symmetries of $C$, i.e. all invertible linear maps $F: \ZZ^6 \rightarrow \ZZ^6$ such that $C \circ F = C$.  The group $\S_3$ of symmetries \eqref{S3} is generated by
\[
\lambda, \mu, \nu \mapsto \mu, \lambda, \nu, \qquad \textrm{ and }
\quad \lambda, \mu, \nu \mapsto \nu^{\square}, \mu, \lambda^{\square},
\]
that correspond in this context to:
\begin{align*}
S: (\lambda_1, \lambda_2 \| \mu_1, \mu_2 \| \nu_1, \nu_2) &\longmapsto ( \mu_1, \mu_2 \| \lambda_1, \lambda_2 \| \nu_1, \nu_2),\\
U: (\lambda_1, \lambda_2 \| \mu_1, \mu_2 \| \nu_1, \nu_2) &\longmapsto (\nu_1 -\nu_3, \nu_1 - \nu_2 \| \mu_1, \mu_2 \| \nu_1, \nu_1 -\lambda_2).
\end{align*}
The additional  involution \eqref{3rectangles} corresponds to 
\begin{multline*}
T: (\lambda_1, \lambda_2 \| \mu_1, \mu_2 \| \nu_1, \nu_2) \\ 
\longmapsto (\lambda_1, \lambda_1 -\lambda_2 \| \mu_1, \mu_1 -\mu_2\| \lambda_1 + \mu_1 - \nu_3, \lambda_1 + \mu_1 - \nu_2)
\end{multline*}
where, again, $\nu_3$ is given by \eqref{nu3}.

Rassart's Chamber Complex and polynomial formulas for the $GL_3$--Littlewood-Richardson coefficients  restrict to a Chamber Complex and polynomial formulas for the function $C$ (set $\lambda_3=\mu_3=0$). The rays of this Chamber Complex are given in Table \ref{table:rays}, and the chambers and polynomial formulas are shown in Table \ref{table:chambers}. 
\begin{table}
\[
\begin{array}{c@{\qquad}c@{\qquad}l}
\textrm{Chamber} & \textrm{Generators} & \textrm{Formula for $C$} \\
\hline
\kappa_{1} &  b, c,  d_1, e_2,  d_2, e_1 & 1 - \lambda_2 - \mu_2 + \nu_1 \\[1mm]
\kappa_{2} &  b, c,  d_1, g_1,  d_2, g_2 & 1 + \nu_2 - \nu_3 \\[1mm]
\kappa_{3} &  b, c,  e_2, g_1,  e_1, g_2 & 1 + \lambda_1 + \mu_1 - \nu_1 \\[1mm]
\kappa_{4} &  b, f,  d_1, e_2,  d_2, e_1  & 1 + \nu_1 - \nu_2 \\[1mm]
\kappa_{5} &  b, f,  d_1, g_1,  d_2, g_2  & 1 + \lambda_2 + \mu_2 - \nu_3 \\[1mm]
\kappa_{6} &  b, f,  e_2, g_1,  e_1, g_2  & 1 - \lambda_3 - \mu_3 + \nu_3 \\[1mm]
\kappa_{7} &  b, c,  d_1, g_1,  d_2, e_1  & 1 + \lambda_3 + \mu_1 - \nu_3 \\
\kappa_{8} &  b, c,  d_1, e_2,  d_2, g_2  & 1 + \lambda_1 + \mu_3 - \nu_3 \\[1mm]
\kappa_{9\ } & b, c,  d_1, e_2,  e_1, g_2  & 1 + \lambda_1 - \lambda_2 \\
\kappa_{10} &  b, c,  e_2, g_1,  d_2, e_1  & 1 + \mu_1 - \mu_2 \\[1mm]
\kappa_{11} &  b, c,  d_1, g_1,  e_1, g_2  & 1 - \lambda_2 - \mu_3 + \nu_2 \\
\kappa_{12} &  b, c,  e_2, g_1,  d_2, g_2  & 1 - \lambda_3 - \mu_2 + \nu_2 \\[1mm]
\kappa_{13} &  b, f,  d_1, g_1,  d_2, e_1  & 1 - \lambda_1 - \mu_3 + \nu_1 \\
\kappa_{14} &  b, f,  d_1, e_2,  d_2, g_2  & 1 - \lambda_3 - \mu_1 + \nu_1 \\[1mm]
\kappa_{15} &  b, f,  d_1, g_1,  e_1, g_2  & 1 + \mu_2 - \mu_3 \\
\kappa_{16} &  b, f,  e_2, g_1,  d_2, g_2  & 1 + \lambda_2 - \lambda_3 \\[1mm]
\kappa_{17} &  b, f,  d_1, e_2,   e_1, g_2  & 1 + \lambda_1 + \mu_2 - \nu_2 \\
\kappa_{18} &  b, f,  e_2, g_1,  d_2, e_1  & 1 + \lambda_2 + \mu_1 - \nu_2 \\[1mm] 
\end{array}
\]
\caption{The Chamber Complex for $C$, from \cite[Table 1]{Rassart}. Note that there is a typo in \cite[Table 1]{Rassart}:  there one should read $\nu_1$ instead of $\nu_3$ in the quasipolynomial formulas for chambers  $\kappa_{13}$ and $\kappa_{14}$.}\label{table:chambers}
\end{table}
 
 \begin{table}
\[
\begin{array}{ccc}
b = (2, 1, 2, 1, 3, 2)
&
c = (1, 1, 1, 1, 2, 1)
&
f = (1, 0, 1, 0, 1, 1)
\\
d_1 = (1, 1, 1, 0, 1, 1) 
&
e_1 = (1, 1, 0, 0, 1, 1)
&
g_1 = (1, 0, 0, 0, 1, 0)
\\
d_2 = (1, 0, 1, 1, 1, 1)
&
e_2 = (0, 0, 1, 1, 1, 1)
&
g_2 = (0, 0, 1, 0, 1, 0)
\end{array}
\]
 \caption{The minimal generators for the rays of the Chamber Complex for $C$.}\label{table:rays}
 \end{table}

\section{Computation of the symmetries}

In this section, we prove Theorem \ref{thm:main}.

\subsection{Reduction to the symmetries of the Chamber Complex}

In order to find all symmetries of the function $C$, we look for symmetries of simpler objects: first we check that any symmetry of $C$ must be a symmetry of the chamber complex. Afterwards, we will show that such a symmetry must be also a symmetry of an even simpler object: the set of the ray generators of the Chamber Complex.

\begin{lemma}\label{lemma1}
Any linear symmetry of $C$ is necessarily a linear symmetry of the chamber complex, i.e. an invertible linear map from $\ZZ^6$ to itself that permutes the cells of the chamber complex.
\end{lemma}
\begin{proof}
We will call $\Cone$ the support of the Chamber Complex, i.e. the cone that is the union of all chambers. 

Let $F$ be a symmetry of $C$. On any chamber $ \kappa$ of the chamber complex, the function $C$ coincides with a polynomial function $P_{ \kappa}$. Let $ \kappa$ be a chamber of the chamber complex. Since $C \circ F=C$ , we have that $C \circ F$ also coincides with $P_{ \kappa}$ on $ \kappa$. Therefore $C$ coincides with $P_{ \kappa} \circ F^{-1}$ on $F( \kappa)$. The cone $F( \kappa)$ can't meet the exterior of the cone $\Cone$. Indeed, if $F( \kappa)$ would meet the exterior of $\Cone$, the intersection would contain a full-dimensional cone of $\RR^6$. But two polynomial functions that coincide on the integer points of a full-dimensional cone must be equal. We would have $P_{ \kappa} \circ F^{-1} =0$, and thus $P_{ \kappa}=0$, which is false. Therefore $F( \kappa) \subset \Cone$. Again, it is impossible that $F( \kappa)$  meet two chambers $\tau_1$ and $\tau_2$ full--dimensionally, because this would imply $P_{\tau_1} = P_{\tau_2}$, and there is no such coincidence in the list of formulas for $C$. As a consequence, there exists one cell $\tau$ such that $F( \kappa) \subset \tau$. Applying the same reasoning to $F^{-1}$, we see that we must have also $F^{-1}(\tau) \subset  \kappa$. As a conclusion, $F( \kappa) = \tau$. Therefore, $F$ permutes the chambers of the chamber complex. It follows that $F$ is a symmetry of the chamber complex.
\end{proof}

\subsection{Reduction to the Symmetries of the rays}


After Lemma \ref{lemma1}, any linear symmetry of $C$ is also a linear symmetry of the chamber complex. Obviously, any linear symmetry of the chamber complex induces also a linear symmetry of the set of its ray generators $\rays=\lbrace b,c,f,d_1, d_2, e_1, e_2, g_1, g_2\rbrace$ (i.e. a permutation of $\rays$ induced by a linear invertible map from $\ZZ^6$ to itself). We will denote with $S(\rays)$ the group of linear symmetries of $\rays$.

We compute with \emph{SageMath} the group $S(\rays)$, thanks to the following lemma. Firstly, note that the group of symmetries of a set of vectors $R=\lbrace v_1, v_2, \ldots, v_n\rbrace$ embeds in the symmetric group $\mathfrak{S}_n$.

\begin{lemma}[{\cite[Proposition 3]{BremnerDutourSchuermann}}]\label{lemma2}
  Let $R=\{v_1, v_2, \ldots, v_n\}$ be a set of vectors of $\ZZ^m$, spanning $\ZZ^m$.
  Then the linear symmetries of $R$ are the automorphisms of the edge--colored complete graph $H(R)$ whose vertices are the $v_i$, with the edge $v_i \-- v_j$ colored by the entry $(i,j)$ of the matrix $Q=  V^t (V V^t)^{-1} V$, where $V$ is the matrix of order $m \times n$ whose columns are the $v_i$.
  

\end{lemma}

In the case under consideration, the edge--colored graph $H(\rays)$ is shown in Figure \ref{graph}, and its automorphism group is readily obtained: $S(\rays)$ is the direct product of the group $\S_2$ of permutations of $\{c, f\}$ and of the wreath product $\S_3 \wr \S_2$ of the permutations of $\{ d_1, e_2, g_1, d_2, e_1, g_2\}$ that stabilize or swap the subsets $\{d_1, e_2, g_1\}$ and $\{d_2, e_1, g_2\}$. In particular, $S(\rays)$ is generated by $v = (c, f)$ , $x = (e_1, g_2)$ and  $y=(e_1, d_2)$ that permute $\{e_1, g_2, d_2\}$, and $s=(d_1, d_2)(e_1, e_2)(g_1,g_2)$ that swaps $\{d_1, e_2, g_1\}$ and $\{d_1, e_1, g_2\}$. 
\begin{figure}
  \includegraphics[width=0.5\textwidth]{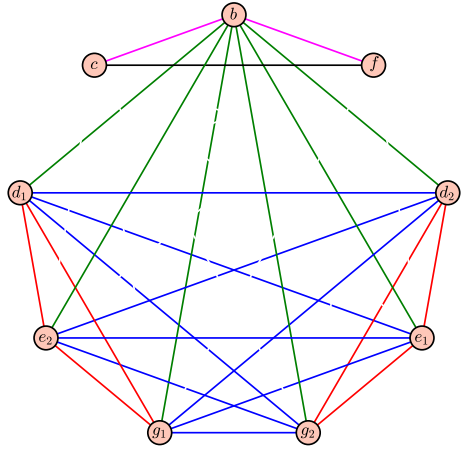}
  \caption{The edge--colored graph $H(\rays)$. Its automorphism group is the group of linear symmetries of the ray generators of the chamber complex.}\label{graph}
\end{figure}

Note that $s$ is precisely the automorphism of $\rays$ induced by the known symmetry $S$ of the Littlewood-Richardson coefficients. Let us consider 
the automorphisms $t,u$ of the rays induced by the other two known symmetries $T$, $U$. It is easily calculated that 
\[
t = (d_1, d_2) (e_1, g_1) (e_2, g_2) (c,f), \quad 
u = (d_1, g_2) (d_2, e_2) (e_1, g_1).
\]
Again with \emph{SageMath}, we check that  $s$, $t$, $u$ and $x$ already generate $S(\rays)$. More precisely, we get that $v = txsx$ and $y=usxsu$, which is easily checked by hand.

\subsection{The symmetries of the rays are also symmetries of the Littlewood--Richardson coefficients}
 
Note that any automorphism of the rays lifts uniquely to a linear automorphism of $\RR^6$. Let $X$ be the lifting of $x$. Let $G$ be the group of all liftings of the elements of $S(\rays)$. After Lemmas \ref{lemma1} and \ref{lemma2}, the group $G$ contains all symmetries of $C$. 
 We will check by explicit calculation that $G$ is actually equal to the group of symmetries of $C$. Since $G$ is generated by $S$, $T$, $U$ and $X$, and that $S$, $T$ and $U$ are known to be symmetries of $C$, it is enough to check that $X$ is a symmetry of $C$. 
 
 To check that $X$ is a symmetry of $C$ we proceed as follows:   for each chamber $ \kappa$, with generators $w_1, w_2, w_3, w_4, w_5, w_6$, we calculate $x(w_1)$, $x(w_2), \ldots$ (This is immediate since $x$ just swaps $e_1$ and $g_2$ and leaves all other ray generators fixed). We check that they are the generators of a chamber $g( \kappa)$ of the Chamber Complex. Then we check by inspection that $g$ is indeed a permutation of the chambers. Last,  we check that for each $ \kappa$, the polynomials $P_{g( \kappa)}$  and $P_{ \kappa} \circ X^{-1}$ are equal.  The latter calculation is performed within \emph{SageMath}, and gives the expected result: $X$ is indeed a symmetry of $C$. Finally, one calculates that $X$ is given by 
 \[
 (\lambda_1, \lambda_2\| \mu_1, \mu_2\| \nu_1, \nu_2) \longmapsto 
 (\lambda_1 + \mu_1 -\nu_2, \lambda_2 + \mu_1-\nu_2\| \nu_2, \mu_2\| \nu_1, \mu_1).
 \] 

\subsection{Transitivity of the action on the chambers.}

By direct inspection of Table \ref{table:chambers}, one can check that the rays of each chamber are: $b$, one of $\{c,f\}$, two of $\{d_1, e_2, g_1\}$ and two of $\{d_2, e_1, g_2\}$. This proves that $G$ permutes transitively the chambers, since $G$ is exactly the stabilizer of $(\{c,f\},\{\{d_1, e_2, g_1\},\{d_2, e_1, g_2\}\})$.

\section{Final remarks}

\subsection{Littlewood--Richardson coefficients associated to $GL_3$}

If we consider the Littlewood--Richardson coefficients associated to $GL_3$, instead of $SL_3$, we get one more generator for the group of symmetries, yielding in total $144 \times 2 = 288$ symmetries. The additional generator sends $(\lambda_1, \lambda_2, \lambda_3 \| \mu_1, \mu_2, \mu_3\|  \nu_1, \nu_2, \nu_3)$ to
\[
(\lambda_1 -m, \lambda_2-m, \lambda_3-m \| \mu_1 + m, \mu_2+m, \mu_3+m\| \nu_1, \nu_2, \nu_3),
\]
where $m=\lambda_3-\mu_3$.
In Rassart's description \cite{Rassart} of the chamber complex for the Littlewood-Richardson coefficients associated to $GL_3$, this generator swaps the additional rays $a_1$ and $a_2$ while fixing all other rays.

\subsection{The case of $SL_N$ for $N \ge 4$}

One finds that the linear symmetries of the Littlewood-Richardson coefficients associated to $SL_N$, for $N$ in $\{4,5,6,7\}$, are only the 12 known symmetries. Indeed, the group of symmetries of these Littlewood--Richardson coefficients embeds in the group of linear symmetries of the ray generators of their support. But one computes that this group has only 12 elements\cite{inprogress}.

\section{Perspectives}

There are two contributions in this work: the results (144 linear symmetries for the  Littlewood-Richardson coefficients associated to the representations of $SL_3$, a much bigger number than expected); and the method (embedding the group of the linear symmetries in the group of symmetries of a chamber complex, and the group of symmetries of its rays). This method may be applied to the study of the symmetries of other families of Littlewood-Richardson coefficients, and even more, to families of other representation--theoretic structural coefficients 
(such as Kostka coefficients, Kronecker coefficients, plethysm coefficients) as long as we have an explicit description of  their piecewise quasipolynomial formulas.  Explicit descriptions of piecewise quasipolynomial for such structural coefficients appear in the literature. For example, for the plethysm coefficients  in the Schur expansion of $s_{\mu}[s_k]$ with $\mu$ any  fixed partition of $3$, $4$ or $5$ \cite{KahleMichalek}, for the Kronecker coefficients $g_{\lambda, \mu, \nu}$ with $\lambda$ and $\mu$ of length $\le 2$ \cite{BriandOrellanaRosas:reducedKronecker} or with $\lambda$, $\mu$, $\nu$ of length $\le 3$ \cite{BaldoniVergneWalter}.


\end{document}